\documentclass[12pt]{amsart}

   \usepackage[colorlinks, pagebackref, linkcolor=red, citecolor=blue]{hyperref}
\usepackage{amsmath,amssymb}
\usepackage{color}
\usepackage{mathrsfs}
\usepackage{graphicx}

\numberwithin{equation}{section}
\theoremstyle{plain}
\newtheorem{theorem}{Theorem}[section]
\newtheorem{lemma}[theorem]{Lemma}
\newtheorem{observation}[theorem]{Observation}

\newtheorem{proposition}[theorem]{Proposition}
\newtheorem{question}[theorem]{Questions}

 {\theoremstyle{definition}}

\newcommand{\zit}[1]{(\ref{#1})}

\def\sing{\operatorname{Sing}}

\def\T{ \mathbb T}
\def\R{ \mathbb R}
\def\H{H^\infty}

\def\D{{ \mathbb D}}
\def\C{{ \mathbb C}}

\def\N{{ \mathbb N}}

\def\e{\varepsilon}

\def\union{\cup}
\def\Union{\bigcup}
\def\inter{\cap}

\def\ov{\overline}

\def\ss{\subseteq}
\def\emp{\emptyset}

\def\buildrel#1_#2^#3{\mathrel{\mathop{\kern 0pt#1}\limits_{#2}^{#3}}}

\def\BP{Blaschke product}

\def\IBP{interpolating Blaschke product}

\begin{document}

\title []{One-component inner functions II}


 \author{Joseph Cima}
\address{\small Department of Mathematics,
UNC,
Chapel Hill, North-Carolina,USA}
\email{cima@email.unc.edu}

 \author{Raymond Mortini}
  \address{
Universit\'{e} de Lorraine\\
 D\'{e}partement de Math\'{e}matiques et  
Institut \'Elie Cartan de Lorraine,  UMR 7502\\
 3, rue Augustin Fresnel\\
 F-57073 Metz} 
 \email{raymond.mortini@univ-lorraine.fr}



\begin{abstract}
We continue our study of the  set $\mathfrak I_c$ of inner functions $u$ in $\H$ with the property that
 there is $\eta\in ]0,1[$
such that the level set $\Omega_u(\eta):=\{z\in\D: |u(z)|<\eta\}$ is connected.  
These  functions are called one-component inner functions.
Here we  show that the composition of two one-component inner functions is again in $\mathfrak I_c$. 
We also give conditions under which a factor of  one-component inner function
 belongs to $\mathfrak I_c$.
\end{abstract}

 \maketitle

\centerline {\small\the\day.\the \month.\the\year} \medskip

\section{Introduction}

One-component inner functions, the collection of which we denote by $\mathfrak I_c$, were first studied by B. Cohn \cite{coh}  in connection with embedding theorems and Carleson-measures.
Recall that an inner function $u$ in $\H$ is said to be a
 {\it one-component inner function} if there is $\eta\in ]0,1[$
such that the level set (also called sublevel set or filled level set) $\Omega_u(\eta):=\{z\in\D: |u(z)|<\eta\}$ is connected. Unimodular constants are considered to belong to $\mathfrak I_c$.
It was shown in \cite[p. 355] {coh} for instance, that arclength on $\{z\in\D: |u(z)|=\e\}$ is  a Carleson  measure whenever 
$$\Omega_u(\eta)=\{z\in\D: |u(z)|<\eta\}$$
is connected and $\eta<\e<1$.  
A detailed  study of the class $\mathfrak I_c$ was undertaken by  A.B. Aleksandrov \cite{alex}.
Classes of explicit examples of one-component inner functions were given by the present authors in \cite{ci-mo}. The most fundamental ones are finite Blaschke products and singluar inner functions $S_\mu$  with finite singularity set (or spectrum), $\sing  S_\mu$. Infinite interpolating Blaschke products with real zeros $(x_n)$ satisfying $0<\eta_1\leq \rho(x_n,x_{n+1})\leq \eta_2<1$ (where $\rho$ is the pseudohyperbolic distance in $\D$)  were also shown to belong to $\mathfrak I_c$.
On the other hand, no finite product of thin interpolating Blaschke products
(these are (infinite) Blaschke products $B$ whose zeros $(z_n)$   satisfy $\lim_n\prod_{k:k\not=n}\rho(z_n,z_k)=1$),  can be in $\mathfrak I_c$.
It also turned out the class of  one-component inner functions is invariant under taking finite products.
In the present note, we are considering when a factor of a one-component inner function is in 
$\mathfrak I_c$ again.  A sufficient criterion is provided. On the other hand, as it is shown,  there  exist two non one-component inner functions $u$ and $v$ such that $uv\in \mathfrak I_c$.
Our main result will show that the class of  one-component inner functions is also invariant under taking
compositions, generalizing special cases dealt with in \cite{ci-mo}.  The results of this note stem from December 2016. Meanwhile (Mai 2018) a manuscript by A. Reijonen \cite{rei}  provides other classes of one-component inner functions.

\section{Main tools}

Our results will mainly be based on the following  known results which we recall for citational reasons.
 
\begin{lemma}\label{compos1}
Given a non-constant  inner function $u$ in $\H$ and $\eta\in \;]0,1[$, 
let  $\Omega:=\Omega_u(\eta)=\{z\in\D: |u(z)|<\eta\}$ be a level set. Suppose that $\Omega_0$ is a
 component (=maximal connected subset)  of $\Omega$.
 Then
\begin{enumerate}
\item[(1)]  $\Omega_0$  is a simply connected domain; that is, $\C\setminus \Omega_0$ has no bounded components.
\item[(2)]  $\inf_{\Omega_0} |u|=0$.
\item[(3)] Either $\ov{\Omega_0}\ss \D$ or $\ov{\Omega_0}\inter \T$ has measure zero.
\end{enumerate}
\end{lemma}
A detailed proof of parts (1) and (2)   is given in \cite{ci-mo}; part (3) is in \cite[p. 733]{be}.

Recall that the spectrum $\sing(u)$ of an inner function  $u$ is the set of all boundary points $\zeta$ for which 
$u$ does not admit a holomorphic extension; or equivalently, for which  $Cl(u,\zeta)=\ov\D$,
where 
$$\mbox{$Cl(u,\zeta)=\{w\in\C: \exists (z_n)\in\D^\N,~ \lim z_n=\zeta$ and $ \lim u(z_n)=w\}$}$$ 
is the cluster set of 
$u$ at $\zeta$ (see \cite[p. 80]{ga}).

The pseudohyperbolic disk of center $z_0\in \D$ and radius $r$ is denoted by $D_\rho(z_0,r)$.

\begin{theorem}[Aleksandrov] \label{alexi}\cite[Theorem 1.11 and Remark 2, p. 2915]{alex}
Let $u$ be an inner function. The following assertions are equivalent:
\begin{enumerate}
\item[(1)] $u\in \mathfrak I_c$.
\item[(2)] There is a constant $C>0$ such that for every $\zeta\in \T\setminus \sing(u)$ we have
$$i)~~~  |u '' (\zeta)|\leq C\; |u' (\zeta)|^2,$$
and 
$$ii)~~~ \mbox{$\liminf_{r\to 1} |u(r\zeta)|<1$ for all $\zeta\in \sing(u)$}.$$
\end{enumerate}
\end{theorem}
Note that, due to this theorem,  $u\in \mathfrak I_c$ necessarily implies 
that $\sing(u)$ has measure zero.

\section{Splitting off factors}
In this section we give a condition under which a factor of a one-component inner function is in
$\mathfrak I_c$ again. Recall from \cite{ci-mo} that for the atomic inner function 
$S(z)=\exp(-\frac{1+z}{1-z})$ and a thin Blaschke product with positive zeros,  
$SB\in \mathfrak I_c$, but not $B$.
For $a\not=0$, let 
$$\phi_a(z)=\frac{|a|}{a}\frac{a-z}{1-\ov az}$$
and $\phi_0(z)=z$. A Blaschke product  $B$ is written as $ B=e^{i\theta}\,\prod_{j=1}^\infty \phi_{a_j}$, where
${\sum_{j=1}^\infty (1-|a_j|)<\infty}$, each $a_j$ appearing as often as its multiplicity needs.
The following result tells us that one can split off finitely many zeros without leaving the class
of one-component inner functions. Any inner function $u$ has the form $u= BS_\mu$,
where $B$ is a Blaschke product and $S_\mu$ a singular inner function 
$$S_\mu(z):=\exp\left(-\int_\T \frac{\zeta+z}{\zeta-z}\;d\mu(\zeta)\right)$$
associated with a positive Borel measure $\mu$ which is singular with respect to Lebesgue measure on $\T$.

\begin{proposition} Let $\Theta\in \mathfrak I_c$ and $a\in \D$. If $\Theta(a)=0$, then 
$v:=\Theta/\varphi_a\in \mathfrak I_c$. 
\end{proposition}
\begin{proof}

Note that  $\Theta=\varphi_a v$. We may assume that $v$ is not constant, otherwise we are done.
Choose $\eta\in \;]0,1[$ so that $\Omega_\Theta(\eta)$ is connected. Let 
$$\delta:=\inf \{|\varphi_a(z)|: |\Theta(z)|=\eta\}.$$
We claim that  $\eta<\delta<1$.  In fact, since the set $L:=\{z\in \D: |\Theta(z)|=\eta\}$ is not empty, 
and $|\varphi|<1$ in $\D$, we see that $\eta<1$. Moreover,
if $z_0\in L$, then 
$$L':=\{|\varphi_a(z)|: |\Theta(z)|=\eta, |\varphi_a(z)|\leq |\varphi_a(z_0)|\}$$
is a compact set in $[0,1]$, and so
$$\inf \{|\varphi_a(z)|: |\Theta(z)|=\eta\}=\inf L'=\min L.'$$
Hence $\delta=|\varphi_a(z_1)|$ for some $z_1\in L$.  Since $v$ is not a unimodular constant, we deduce from $|\Theta(z_1)|=|\varphi_a(z_1)|\; |v(z_1)|$ that $\eta<\delta$.

Consequently,  if $|\Theta(z)|=\eta$, 

\begin{equation}\label{cb-level}
|v(z)|=\frac{|\Theta(z)|}{|\varphi_a(z)|}\leq \frac{\eta}{\delta}:=\eta'<1.
\end{equation}
We claim that
$$\Omega_v(\eta)\ss \Omega_\Theta(\eta)\ss \Omega_v(\eta').$$
Notice that  the first inclusion is obvious.  To verify the second inclusion, let $z_0\in \Omega_\Theta(\eta)$.
We discuss three cases:  $\rho(z_0,a)<\delta, \rho(z_0,a)=\delta$ and $\rho(z_0,a)>\delta$.

To this end, we first note that $D_\rho(a,\delta)\ss \Omega_\Theta(\eta)$.  In fact, 
if $\rho(a,z)=|\varphi_a(z)|< \delta$, then $|\Theta(z)|<\eta$, since otherwise
$\Theta(a)=0$ implies the existence of  $z_0\in D_\rho(a,\delta)$ with $|\Theta(z_0)|=\eta$ and so,
by the definition of $\delta$, $|\varphi_a(z)|\geq \delta$. An obvious contradiction. 

 Hence $|\Theta(z)|\leq \eta$ for $\rho(z,a)=\delta$. Thus \zit{cb-level} holds true for
$z\in \partial D_\rho(a,\delta)$. By the maximum principle, 
$|v(z)|< \eta'$ on  $D_\rho(a,\delta)$. 
If $\rho(z,a)\geq \delta$ and  $|\Theta(z)|<\eta$,  then,
as in \zit{cb-level}, $|v(z)|<\eta'$, too. 
We deduce that $\Omega_\Theta(\eta)\ss \Omega_v(\eta')$.

Now we are able to prove that $\Omega_v(\eta')$ is connected. Assuming the contrary,
there would exist a component $\Omega_1$ of $\Omega_v(\eta')$ distinct (and so disjoint) from that containing the connected set
 $\Omega_\Theta(\eta)$. In particular, $|v|\geq |\Theta|\geq \eta$ on $\Omega_1$.  By Lemma  \ref{compos1},
 $\inf_{\Omega_1} |v| =0$; an obvious contradiction.
 \end{proof}

 The preceding result admits the following generalization.

 \begin{proposition}\label{main}
   Let $u,v$   be two non-constant inner functions and put $ \Theta=uv$. Suppose that 
  \begin{enumerate}
\item[(i)] $\Theta\in \mathfrak I_c$  and that $\eta\in ]0,1[$ is chosen so that $\Omega_\Theta(\eta)$ is connected.
\item[(ii)] $\sigma:=\sup_{|\Theta|=\eta} |v|\in \;]\eta,1[$  (or equivalently, 
$\delta:=\inf_{|\Theta|=\eta} |u|\in \;]\eta,1[$).
 \end{enumerate}
Then $v\in \mathfrak I_c$. The assertion does not necessarily hold if $\sigma=1$ (or, equivalently, if $\delta=\eta$).
\end{proposition}
\begin{proof}
Due to  hypothesis  (ii),   we have the following estimate on $|\Theta|=\eta$:
\begin{equation}\label{defdelta}
|u|= \frac{|\Theta|}{|v|}\geq \frac{\eta}{\sigma}=\delta.
\end{equation}
Note that $\delta\in\;]\eta, 1[$.
We claim that 
\begin{equation}\label{ooo}
\Omega_u(\delta)\ss \Omega_\Theta(\eta)\inter \Omega_v(\sigma).
\end{equation}

To this end,  we first show that $|\Theta|<\eta$ on $\Omega_u(\delta) $.  In fact, assuming the contrary, 
there exist $z_0\in \Omega_u(\delta)$ such that $|\Theta(z_0)|\geq \eta$. 
Let $\Omega_0$ be that component of $\Omega_u(\delta)$ containing $z_0$.
By Lemma \ref{compos1} (2), $\inf_{\Omega_0} |u|=0$.  Since $u$ is a factor of $\Theta$, we conclude 
that there  exists $z_1\in \Omega_0\ss \Omega_u(\delta)$ such that $|\Theta(z_1)|<\eta$.
Thus, the connected set $\Omega_0$ meets $\{|\Theta|<\eta\}$ as well as its complement.
Hence $\Omega_0$ meets the topological boundary of $\Omega_\Theta(\eta)$.
Because $\Omega_0\ss\D$, we obtain $z_2\in \Omega_0$ such that $ |\Theta(z_2)|=\eta$.
Hence, by  (ii), $|v(z_2)|\leq \sigma$  and so $|u(z_2)|\geq \delta$
by \zit{defdelta}.
Both assertions $|u(z_2)|\geq \delta$ and $z_2\in \Omega_0\ss \Omega_u(\delta)$  cannot hold.
Thus our assumption right at the beginning of this paragraph was wrong. We deduce that
\begin{equation}\label{omega1}
\Omega_u(\delta)\ss \Omega_\Theta(\eta).
\end{equation}
By continuity, this inclusion implies  that $|\Theta|\leq \eta$ on $\{|u|=\delta\}$. Hence, for $|u(z)|=\delta$, 
\begin{equation}\label{omo2}
 |v(z)|=\frac{|\Theta(z)|}{|u(z)|}\leq \frac{\eta}{\delta}\buildrel=_{}^{\zit{defdelta}}\sigma.
 \end{equation}
Now $\partial \Omega_u(\delta)\inter \D=\{|u|=\delta\}$. If $\Omega$ is a component  of 
$\Omega_u(\delta)$ whose closure belongs to $\D$, then by the maximum principle and \zit{omo2},
$|v|< \sigma$ on $\Omega$. If $E:=\ov \Omega\inter \T\not=\emp$, then  $E$ has measure zero 
by Lemma \ref{compos1} (3).  The maximum principle with exceptional points (see \cite[p. 729] {be} or \cite{gar})
now implies that $|v|<\sigma$ on $\Omega$. Consequently, 
\begin{equation}\label{oop}
\Omega_u(\delta)\ss \Omega_v(\sigma).
\end{equation}
Thus \zit{ooo} holds.  Next we will deduce  that
\begin{equation}\label{oo}
\Omega_v(\eta)\ss \Omega_\Theta(\eta)\ss \Omega_v(\sigma).
\end{equation}
To see this,  observe that the  first inclusion is obvious because $v$ is a factor of $\Theta$. 
To prove the second inclusion, we write the $\eta$-level set of $\Theta$ as
$$\Omega_\Theta(\eta)=\Big(\Omega_\Theta(\eta)\inter \Omega_u(\delta)\Big)\union
\Big(\Omega_\Theta(\eta)\setminus \Omega_u(\delta)\Big).$$
By \zit{oop}, the first set in this union is contained in $\Omega_v(\sigma)$. The second set is also contained in $\Omega_v(\sigma)$, because 
if $|u(z)|\geq \delta$ and $z\in \Omega_\Theta(\eta)$,   then 
\begin{equation}\label{omo3}
 |v(z)|=\frac{|\Theta(z)|}{|u(z)|}<\frac{\eta}{\delta}\buildrel=_{}^{\zit{defdelta}}\sigma.
 \end{equation}

To sum up, we have shown that for every $z\in \Omega_\Theta(\eta)$ we have
$|v(z)|<\sigma$ both in the case where $|u(z)|<\delta$  and $|u(z)|\geq\delta$.
Thus
$$\Omega_\Theta(\eta)\ss \Omega_v(\sigma),$$
and so, \zit{oo} holds.
Using these inclusions \zit{oo},  we are now able to prove that $\Omega_v(\sigma)$ is connected. Assuming the contrary,
there would exist a component $\Omega_1$ of $\Omega_v(\sigma)$, distinct (and so disjoint) from that containing the connected set
 $\Omega_\Theta(\eta)$. In particular, $|v|\geq |\Theta|\geq \eta$ on $\Omega_1$.  By Lemma \ref{compos1} (2),
 $\inf_{\Omega_1} |v| =0$; an obvious contradiction.

Finally we construct an example showing that in (ii) the parameter $\sigma$ cannot be taken to be 1. In fact, let
$v$ be a thin \IBP\ with positive zeros clustering at 1, for example
$$v(z)=\prod_{n=1}^\infty \frac{1-1/n!-z}{1-(1-1/n!)z},$$
 and let $u(z)=S(z):=\exp[-(1+z)/(1-z)]$ be the atomic inner function.
Then, by  \cite[Proposition  2.8]{ci-mo},  $\Theta=uv\in \mathfrak I_c$. However, $v\notin \mathfrak I_c$,
\cite[Corollary 21]{ci-mo}.  Thus, by the main assertion of this Proposition,   $\sigma=\sup_{|\Theta|=\eta}|v|=1$.  (A direct proof of the assertion $\sigma$ can also be given 
using \cite[p.~55]{gs}), by noticing that the boundary of the component $\Omega_\Theta(\eta)$ is a closed curve in $\D\union \{1\}$.)
\end{proof}

\begin{observation}\rm{
We know from \cite[Proposition 2.9]{ci-mo}  that $u,v\in \mathfrak I_c$ implies $uv\in \mathfrak I_c$. Here is an example showing that neither $u$ nor $v$ must belong to $\mathfrak I_c$ for $uv$ to be in $\mathfrak I_c$. In fact, let
$b$ be a thin Blaschke product  with real zeros clustering at $1$ and $-1$ (just consider $b(z)=v(z)v(-z)$, $v$ as above). 
Let $\tilde u:=Sb$ and $\tilde v(z):=S(-z)b(z)$. Then $\Theta:=\tilde u\tilde v\in\mathfrak I_c$, because
$\Theta(z)=\big(S(z) v^2(z)\big) \big(S(-z) v^2(-z)\big)$ is the product of two functions in $\mathfrak I_c$ (same proof as in
\cite[Proposition  11]{ci-mo}), but neither  $\tilde u$ nor $\tilde v$ belong to $\mathfrak I_c$. This can be seen as follows:
since $S(-1)=1$, $\tilde u=Sb$ behaves as $b$ close to $-1$. Thus, for  $\eta$ arbitrarily close to 1, the level set 
$\Omega_{\tilde u}(\eta) $ is contained in a union of pairwise disjoint pseudohyperbolic disks $D_\rho(x_n, \eta^*)$, $n=0,1,2,\cdots$, together with  some tangential disk $D$ at $1$, 
where $x_0=0$ and $x_n$ is the $n$-th negative zero of $b$ (this works similarily as in  \cite[Corollary 21]{ci-mo}
and \cite[Proposition 11]{ci-mo}).
}
\end{observation}

\section{Composition of  one-component inner functions}

In \cite{ci-mo} we showed that for every finite \BP\ $B$ and $\Theta\in \mathfrak I_c$, the compositions  $S\circ B\in \mathfrak I_c$ and $B\circ \Theta\in \mathfrak I_c$.   Using the following standard Lemma \ref{bder}, we will extend this to arbitrary one-component inner functions.

\begin{lemma}\label{bder}
1) Let $B$ be a Blaschke product with zero sequence  $(a_n)_{n\in\N}$. Then  the following
inequalities hold for  $\xi \in\T\setminus \sing(B)$:
$$|B'(\xi)|=\sum_{n\in\N}\frac{1-|a_n|^2}{|a_n-\xi|^2}\geq
 \frac{1-|a_{n_0}|}{1+|a_{n_0}|}>0, ~ \forall n_0\in \N.$$
 2) If $u$ is an inner function  for which $\sing(u)\not=\T$, then 
 $$\delta_u:=\inf \{|u'(\xi)|:  \xi \in \T\setminus \sing(u)\}>0.$$
\end{lemma}
\begin{proof}
1) Just compute the logarithmic derivative $B'/B$ and note that on $\T\setminus \sing(B)$
the \BP\ $B$ converges.

2) Let $\varphi_a(z)=(a-z)/(1-\ov a z)$.
 By Frostman's theorem (see \cite[p. 79]{ga}) there is $a\in\D$ such that 
$B:=\varphi_a\circ u$ is a Blaschke product. Of course, $\sing(u)=\sing(B)$, 
$u=\varphi_a\circ B$ and $\varphi_a'(z)=- (1-|a|^2)/(1-\ov az)^2$.
 Hence, for $\xi\in \T\setminus \sing(u)$,
\begin{eqnarray*}
|u'(\xi)|&=&|\varphi_a'(B(\xi))|\; |B'(\xi)|\geq \frac{1-|a|^2}{|1-\ov a B(\xi)|^2}\; \delta_B\\
&\geq& \frac{1-|a|}{1+|a|}\;\delta_B>0.
\end{eqnarray*}
\end{proof}

\begin{theorem} \label{compo}
 If $u$ and $v$ are  two non-constant inner functions  in $\mathfrak I_c$, then 
 $u\circ v\in\mathfrak I_c$.

\end{theorem}
\begin{proof}
As in \cite{ci-mo}, we shall use Aleksandrov's criterion \ref{alexi}.

(1) Let $\Theta:=u\circ v$. It is well known that $\Theta$ is an inner function again (see e.g. 
\cite[p.83]{pav}). Now
$$\sing(\Theta)= \sing(v) \union \{\xi\in \T\setminus \sing(v): v(\xi)\in \sing(u)\}.$$
Since $v\in \mathfrak I_c$, $\liminf_{r\to 1} |v(r\zeta)|<1$ for every $\zeta\in\sing(v)$ (Theorem \ref{alexi}). Say $v(r_n\zeta)\to w_0\in \D$.   Then
\begin{equation}
\Theta(r_n\zeta)=u(v(r_n\zeta))\to u(w_0)\in \D.
\end{equation}
If $\xi\in \sing(\Theta)\setminus \sing(v)$, then $v(r\xi)\to v(\xi)= e^{i\theta}\in\sing(u) $ for some $\theta\in  \R$.
By Lemma \ref{bder}, $v'(\xi)\not=0$; hence $v$ is a conformal map in small neighborhoods of $\xi$;
in particular, due to the angle conservation law, the curve $\gamma:r\mapsto v(r\xi)$ stays in a cone with
 aperture $0<2\sigma<\pi$ and  cusp
at $e^{i\theta}\in \sing u$. Since $u\in \mathfrak I_c$,  $\liminf |u(re^{i\theta})|<1$. 
We claim that $\liminf |u(v(r\xi))|<1$, too. To see this, choose a pseudohyperbolic radius $\rho_0$
 so big that for some $r_0\in\; ]0,1[$ the cone 
 $$C:=\{z\in \D: |z|\geq r_0, |\arg z- \theta|<\sigma\}$$ 
is entirely contained in the domain
$$V:=\Union_{-1<x<1}D_\rho(x,\rho_0).$$
        
 \begin{figure}[h]
  \scalebox{.40} 
  {\includegraphics{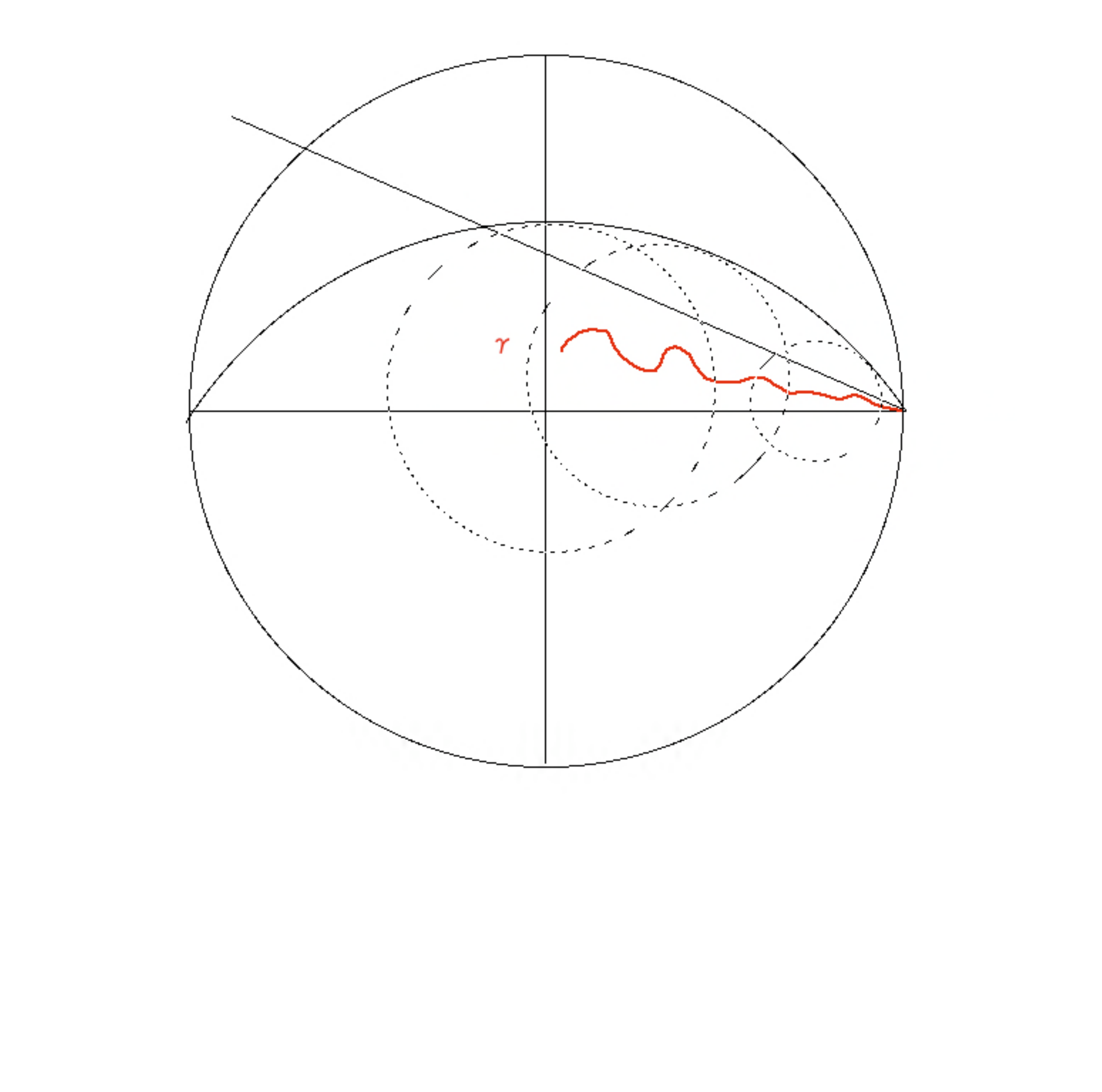}} 
\vspace{-3.5cm}
\end{figure}

Note that by \cite{moru}, the boundary of $V$ is the union of two arcs of circles  cutting
the line $\{se^{i\theta}: s\in \R\}$ at $±e^{i\theta}$ under  an angle $\alpha$ with $\sigma <\alpha <\pi/2$ (see the figure, where we sketched the situation for $\theta=0$).

Choose $r_n$ so that $\lim u(r_n e^{i\theta})=a\in \D$. 
Then the curve $\gamma$ cuts the boundary of infinitely many disks $D(r_ne^{i\theta},\rho_0)$ twice.
But for $z\in D(r_ne^{i\theta},\rho_0)$  we have
$$\frac{|u(z)|-|u(r_ne^{i\theta})|}{1-|u(r_ne^{i\theta})|\, |u(z)|}\leq\rho(u(z), u(r_ne^{i\theta}))\leq \rho(z,r_ne^{i\theta})\leq \rho_0,$$
and so
$$|u(z)|\leq \frac{\rho_0+|u(r_ne^{i\theta})|}{1+|u(r_ne^{i\theta})|\, \rho_0}.$$
This clearly implies that
$$\liminf |u(v(r\xi))|<1.$$
Consequently, $\liminf|\Theta(r\xi)|<1$ for every $\xi\in \sing(\Theta)$.
Next we verify the first condition in Aleksandrov's theorem.

\begin{eqnarray}\label{compos}
A:=\frac{(u\circ v)''}{[(u\circ v)']^2}&=&\frac{u''\circ v}{(u'\circ v)^2}+ \frac{(u'\circ v)}{(u'\circ v)^2}\frac{v''}{{v\,}'^{\,2}}\\\nonumber
&=&\frac{u''\circ v}{(u'\circ v)^2}+ \frac{1}{u'\circ v}\frac{v''}{{v\,}'^{\,2}}.
\end{eqnarray}
If $\zeta\in \T\setminus \sing(u\circ v)$, then  $|v(\zeta)|=1$  and $\xi:=v(\zeta)\not\in \sing(u)$.
Since $u,v\in \mathfrak I_c$, we deduce from Lemma \ref{bder} and Aleksandrov's theorem
\ref{alexi} that  
\begin{eqnarray*}
|A(\zeta)|&\leq& \sup_{\beta\notin\sing u} \frac{|u''(\beta)|}{|u'(\beta)|^2}+ \frac{1}{\delta_u}\sup_{\alpha\notin \sing v}\frac{|v''(\alpha)}{|v'(\alpha)|^{\,2}}=:C<\infty,
\end{eqnarray*}
where  $$\delta_u:=\inf \{|u'(\xi)|:  \xi \in \T\setminus \sing(u)\}.$$
Hence $\Theta\in \mathfrak I_c$.
\end{proof}


\begin{theorem}
1) Let $E\ss\T$ be a closed finite  set. Then 
there exists a one-component inner function $u$ such that for 
some $\eta_0\in \;]0,1[$ (and hence for all $\eta\in\; [\eta_0,1[$) the associated level set $\Omega_u(\eta)$ is connected and 
has the property that 
\begin{equation}\label{spec-lev}
\ov{\Omega_u(\eta)}\inter \T=\sing(u)=E.
\end{equation}
2) There exists $u\in \mathfrak I_c$ such that $\ov{\Omega_u(\eta)}\inter \T=\sing(u)$ is
an infinite set.
\end{theorem}

\begin{proof}
1) Let $E=\{\lambda_1,\dots, \lambda_N\}$ be finite. Then the function
$S_\mu$ given by
$$S_\mu(z)=\prod_{j=1}^N \exp\left(- \frac{\lambda_j+z}{\lambda_j-z}\right)$$
belongs to $\mathfrak I_c$ (by \cite[Corollary 17]{ci-mo}) and satisfies \zit{spec-lev}.

2) Let $E=S^{-1}(1)$ be the countably infinite set of points where the atomic inner function
$S(z)=\exp(-(1+z)/(1-z))$ takes the value $1$, and let $b$ be the \IBP\ with zeros $1-2^{-n}$.
Then $b$ and $S$ belong to $\mathfrak I_c$ (see \cite[Theorem 6]{ci-mo}). By Theorem
\ref{compo}, $u:=b\circ S\in\mathfrak I_c$. It is easy to see that 
$\ov{\Omega_u(\eta)}\inter \T=\sing(u)=E$.  The same holds true for $S\circ b$ as well; 
just note that the argument function of $b$ on $\T\setminus \{1\}$ is unbounded  when
 approaching $1$ from  both sides  on the circle (see \cite[p. 92]{ga}), so that $b^{-1}(\{1\})$ is infinite.
Thus we have a singular inner function in $\mathfrak I_c$ with infinitely many singularities.
\end{proof}

Finally, let us mention that for 
 inner functions $u$, we always have
$$\sing u=\ov{\Omega_u(\eta)}\inter \T,$$
$0<\eta<1$.

\begin{question}\rm
i) Given any countable closed subset $E$ of $\T$, does there exist $u\in\mathfrak I_c$ such that 
$\ov{\Omega_u(\eta)}\inter \T=E$?

ii) Is the set $\ov{\Omega_u(\eta)}\inter \T$ necessarily countable whenever $u\in \mathfrak I_c$?
We don't think so. As indicated by Carl Sundberg \cite{su}, the usual Cantor ternary set may be the support of some singular measure $\mu$
whose associated singular inner function $S_\mu$ belongs to $\mathfrak I_c$. 

iii) Give a description of those closed subsets $E$ of $\T$  such that for some singular inner function $S_\mu$ with $\sing S_\mu=E$  every inner factor  of $S_\mu$ belongs to $\mathfrak I_c$.

For example, finite subsets of $\T$ have this property \cite[Corollary 17]{ci-mo}) .
\end {question}

{\bf Acknowledgements}\\
 
 Le deuxi\`eme  auteur remercie l'UFR MIM (Math\'ematiques-Informatique-M\'ecanique) de l'universit\'e de Lorraine, Campus Metz,  pour lui avoir donn\'e la possibilit\'e, via une mise \`a disposition,  d'enseigner \`a l'Universit\'e du Luxembourg. Un merci aussi \`a l'Universit\'e du Luxembourg pour avoir soutenu cette d\'emarche.

 \end{document}